\documentclass[10pt,reqno]{amsart}
\usepackage{amsmath}
\usepackage{amssymb}
\usepackage{amsthm}
\usepackage{verbatim}
\usepackage{graphicx}
\usepackage{marvosym}
\usepackage{color}
\usepackage[table]{xcolor}
\usepackage{soul}
\usepackage{amsrefs}

\newtheorem{theorem}{Theorem}[section]

\newtheorem{definition}[theorem]{Definition}

\newtheorem{problem}[theorem]{Problem}
\newtheorem{remark}[theorem]{Remark}

\usepackage{color}
\usepackage[table]{xcolor}
\usepackage{soul}

\numberwithin{equation}{section}

\begin{document}\title[Universal Block Tridiagonalization]{
Universal Block Tridiagonalization in $\mathcal{B}(\mathcal{H})$ and Beyond}

\author{Sasmita Patnaik}
\address{S. Patnaik, Department of Mathematics and Statistics, Indian Institute of Technology, Kanpur 208016 (INDIA)}
\email{sasmita@iitk.ac.in}
\author{Srdjan Petrovic}
\address{S. Petrovic, Department of Mathematics, Western Michigan University, Kalamazoo, MI 49008 (USA)}
\email{srdjan.petrovic@wmich.edu}
\author{Gary Weiss*}
\address{G. Weiss, Department of Mathematics, University of Cincinnati, Cincinnati, OH 45221-0025 (USA)}
\email{gary.weiss@math.uc.edu}

\thanks{Partially supported by Simons Foundation grant number 245014\\In memoriam of Victor Lomonosov}
\dedicatory{In memoriam of Victor Lomonosov}

\maketitle

\begin{abstract}
For $\mathcal{H}$ a separable infinite dimensional complex Hilbert space, we prove that every  $\mathcal{B}(\mathcal{H})$ operator has a  basis with respect to which its matrix representation has a universal block tridiagonal form with block sizes given by a simple exponential formula independent of the operator. {From this, such a matrix representation can be further sparsified to {slightly sparser forms}; it can lead to a direct sum of even sparser forms reflecting in part some of its reducing subspace structure; and in the case of operators without invariant subspaces (if any exists), it gives a plethora of sparser block tridiagonal representations.}
{An extension to unbounded operators occurs for a certain domain of definition condition.} Moreover this process gives rise to many different choices of block sizes. 
\end{abstract}

\section{Introduction}

How sparse can a matrix of an operator be? By this we mean the following: If $T$ is a bounded linear operator on infinite dimensional, separable, complex Hilbert space $\mathcal{H}$, can we find an orthonormal basis with respect to which the matrix of $T$ has as many zero entries as possible? 
A change of basis corresponds to a unitary operator $U$, which yields the matrix  representation of $T$ in the new basis $\{Ue_n\}$, or equivalently,  the matrix of $U^{-1}TU$ in the original basis $\{e_n\}$.
 Thus our question can be phrased as: How sparse can $U^{-1}TU$ be made? 

An extreme example is the Spectral Theorem for normal compact operators yielding diagonal operators.
It is well known that  every selfadjoint operator (even when not compact) that possesses a cyclic vector can be represented as a tridiagonal matrix, and if it does not possess a cyclic vector then it can at least be represented as a direct sum (finite or infinite) of tridiagonal matrices.
We found a way to universally generalize these tridiagonal phenomena but with block tridiagonal matrices. That is, we will extend the tridiagonal form idea to each $\mathcal{B}(\mathcal{H})$ operator but it will be  a block tridiagonal matrix with universal block sizes independent of the operator. Moreover, our methods will hold  more generally, i.e. for all $\mathcal{B}(\mathcal{H})$ operators and for unbounded operators with a certain constraint.

{Block tridiagonal matrices have been useful in establishing various results related to \cite{PT1971}, the Pearcy--Topping 
compact commutator problem: What operators are commutators of compact operators, i.e., operators of the form
$[A,B]=AB-BA$, with $A,B$ compact? An outstanding result in this direction is Anderson's construction in \cite{JA1977}, where he 
employed block tridiagonal matrices with a particular arithmetic mean growth
to
prove that every rank one projection  operator is a commutator of compact operators. 
From here, he proved the important consequence: Every compact operator is  a commutator of a compact operator with a bounded operator.}
{Moreover, in \cite{PT1971} Pearcy--Topping asked whether every trace class operator with zero trace is a commutator of Hilbert-Schmidt operators, or at least a finite sum of such commutators. In the same period the third author 
answered these questions in the negative in \cite{GW1980}. This work introduced the 
study of matrix sparsification in terms of 
staircase form representations.}

Historically, \cite{GW1980} introduced staircase forms for general operators (Theorem~\ref{t1}) which herein leads  us to  universal block tridiagonal forms (Theorem~\ref{T1}).
In this article, we will obtain a general matrix sparsification via a special unitary operator by showing how the staircase forms can be reorganized into block tridiagonal forms. 
As in Anderson's model \cite{JA1977}, we hope block tridiagonal forms are more manageable for computations. 
Then we provide two independent further sparsifications, potentially  even more manageable for computations.

\section{Sparsifying arbitrary matrices}\label{sec2}

{The following general staircase form is obtained (from a slight modification of \cite[Lemma]{GW1980})
by considering the free semigroup on the two generators $T,T^*$ 
with any basis $\{e_n\}$ to generate a new basis via applying Gram-Schmidt to 
$e_1$, $Te_1$, $T^*e_1$, $e_2$, $T^2e_1$, $T^*Te_1$, $e_3$, $TT^*e_1$, ${T^*}^2e_1$, $e_4$, $Te_2$, $T^*e_2$, $e_5$, $\dots$.} {The latter list consists of $\mathcal{W}(T,T^*)$ (all words in $T,T^*$, including the empty word $I$) applied to all elements of $\{e_n\}$ and arranged in this special order: first list $e_1$ followed by $T,T^*$ applied to $e_1$ to obtain the first three vectors; then list $e_2$, followed by $T,T^*$ applied to the second vector in this list, namely $Te_1$; and so on.}

An alternate way to view this is to start an induction with the first three terms $e_1$, $Te_1$, $T^* e_1$, followed by $e_2$, $e_3$, $\dots$, and then use Gram-Schmidt to obtain $f_1,f_2,f_3$ (the first 3 vectors of the new basis). Inductively, assume that $f_1,f_2,\dots, f_{3n}$ have been chosen orthonormal with $e_1,e_2,\dots,e_n$ in their span and $e_i, Tf_i, T^* f_i\in \vee_{j=1}^{3n} f_j$, for $1\le i\le n$;
then we extend this list to
$ f_{3n+1},f_{3n+2}, f_{3(n+1)}
$
by continuing the Gram-Schmidt process with the first three new linearly independent vectors from the list 
$$
f_1,f_2,\dots,f_{3n},
e_{n+1},Tf_{n+1}, T^*f_{n+1},e_{n+2},e_{n+3},\dots. 
$$
Clearly because $\{e_n\}$ spans $\mathcal{H}$, $\{f_n\}$ forms a basis. And moreover, this basis together with this inductive condition  yield the following matrix form with respect to $\{f_n\}$ by using the condition $Tf_n, T^*f_n\in \vee_{j=1}^{3n} f_j$ for all $n$.

\begin{theorem}\cite[Theorem 2]{GW1980}
\label{t1}
For every $T \in B(H)$, there is a basis whose implementing operator $U$ fixes an arbitrary $e_1$ and
with respect to which basis $\{e_n\}$ the matrix of $U^{-1}TU$ takes the staircase form 
\begin{equation}\label{E0}
U^{-1}TU = \begin{pmatrix}
*&*&*&0&\cdots\\
*&*&*&*&*&*&0& \cdots\\
*&*&\cdots\\
0&*&\\
0&*&\\
0&*&\\
\vdots&0&
\end{pmatrix},
\end{equation}
with each $e_n\in \vee_{k=1}^{3n} Ue_k$ and where this (\ref{E0}) matrix has row and column support lengths $3,6,9,\dots$.

In addition, if $S_1, \cdots, S_N$ is any finite collection of selfadjoint operators, then there is a unitary operator $U$ that fixes $e_1$ for which each operator $U^{-1}S_kU$ has the form (\ref{E0})  with $3,6,9,\dots$ replaced by $N+1, 2(N+1), 3(N+1), \cdots$, with each $e_n\in\vee_{k=1}^{1+(n-1)(N+1)} Ue_k$. When $S_1, \cdots, S_N$ are not necessarily selfadjoint, we have $3,6,9,\dots$ replaced by $2N+1, 2(2N+1), 3(2N+1), \cdots$, with each $e_n\in\vee_{k=1}^{1+(n-1)(2N+1)} Ue_k$. 
\end{theorem}

\begin{definition}
We call the $*$-entries in (\ref{E0}) the support entries. Albeit some can also be zero as we shall see in Theorems~\ref{T2} and \ref{t3}.
\end{definition}
\begin{remark}\label{rem1}
(i) In fact we have a little more. 
The proof for a single operator yields $2,5,8,\dots$ for the columns  and $3,6,9,\dots$ for the rows, i.e., $Tf_n$ is a linear combination of at most $f_1, f_2, \cdots, f_{3n-1}$ vectors and $T^*f_n$ is a linear combination of at most $f_1, f_2, \cdots, f_{3n}$ vectors, for $n\geq 1$. From here an important question is: {Is there a more substantial sparsification of the form in (\ref{E0})?} Theorems \ref{T2} and \ref{t3} achieve this and imply the obvious question: Can we sparsify these forms even further? One goal as mentioned earlier is to improve their computation potential.

 {When $T$ possesses a cyclic vector $v$ (i.e., $v, Tv, T^2v,\dots$ spans $\mathcal{H}$), (\ref{E0}) would have instead column support sizes $2,3,4,\dots$ obtaining an upper Hessenberg form \cite[Problem 44]{PH1982}. Or if one preferred also simultaneous sparsification of the rows, and had a joint cyclic vector in that $v, Tv, T^*v, T^2v, T^*Tv, TT^*v, {T^*}^2v, \dots$ span $\mathcal{H}$ (the free semigroup on the two generators $T,T^*$ but arranged in this specific order),  then applying Gram-Schmidt to this sequence obtains the matrix pattern $2,4,6,\dots$  columns and $3,5,7, \dots$ rows. So in particular, for an operator $T$ with no invariant subspaces (if indeed one exists), every nonzero vector is cyclic (otherwise for noncyclic $v$, the span of 
 $v, Tv, T^2v, \dots$ is a nontrivial invariant subspace). 
  Or, if $T$ has  no proper reducing subspaces (which can occur) then every vector is  jointly cyclic, i.e. $v, Tv, T^*v, T^2v, T^*Tv, TT^*v, {T^*}^2v, \dots$   spans $\mathcal{H}$. In either case we obtain for $T$ this above {$2,4,6,\dots$/ $3,5,7,\dots$}
  sparser staircase pattern.}

   {(ii) The same applies to those unbounded operators $T$ whose free semigroup on these two generators has a basis on which all words ${w}(T,T^*)$ are defined.} Albeit, we don't know how to test for this.
\end{remark}

\noindent\textbf{From Staircase to Block Tridiagonal Matrix Forms}

\medskip
Theorem \ref{t1} gives a striking $3,6,9,\dots$ staircase universal form for an arbitrary operator and  universal simultaneous staircase forms with larger stairs for arbitrary finite collections. 

Staircase form (\ref{E0}) will allow us to represent a matrix in universal block tridiagonal form
\begin{equation}\label{E}
T = \begin{pmatrix}
C_1 &  A_1  &0           \\
B_1 & C_2&A_2              \\          
 0 &B_2 &C_3 &\ddots            \\
&\, &\ddots  &\ddots    
\end{pmatrix},
\end{equation}
which we believe may be fundamental and of broad general interest.
Even for our  orthonormal basis $\{e_n\}_{n=1}^{\infty}$ in which $T$ is given by {(\ref{E})},
the sizes of these blocks we will determine and see they are not unique (see Theorem~\ref{T1}). We believe it could be of further general  interest to make the (\ref{E}) 
matrix blocks as sparse as possible, i.e., obtain further universal zeros.
Two ways to accomplish this are demonstrated in Theorems~\ref{T2}--\ref{t3}.
We begin with the
following theorem which gives canonical dimensions for the blocks in (\ref{E}) in order that they cover all the support entries. And since there will be no change of basis, we  retain the Theorem~\ref{t1} condition that each $e_n\in \vee_{k=1}^{3n} Ue_k$.

\begin{theorem}\label{T1}
For all $T \in B(H)$, the block tridiagonal partition of the matrix of $T$ induced by (\ref{E0}) is given by (\ref{E}),
where $C_1, A_1$, and $B_1$ are $1\times 1$, $1 \times 2$ and $2 \times 1$ matrices, respectively; and for $k\geq 2$,
\begin{equation}\label{eq1}
\begin{aligned}
C_{k}&: 2(3^{k-2}) \times 2(3^{k-2}) \mbox{ square matrix (i.e.,  $2\times 2,6\times 6, 18\times 18,\dots$)}\\
A_{k}&: 2(3^{k-2}) \times 2(3^{k-1}) \mbox{ rectangular  wide matrix}\\
B_{k}&: 2(3^{k-1}) \times 2(3^{k-2}) \mbox{ rectangular tall matrix.}
\end{aligned}
\end{equation}
Alternatively one can choose
$C_1,A_1,B_1$ to be $n_1\times n_1$, $n_1\times 2n_1$, $2n_1\times n_1$ matrices respectively; and for $k\ge 2$, $C_k,A_k,B_k$ respectively of sizes $2(3^{k-2})n_1 \times 2(3^{k-2})n_1$, $2(3^{k-2})n_1 \times 2(3^{k-1})n_1$, $2(3^{k-1})n_1 \times 2(3^{k-2})n_1$.
More generally, necessary and sufficient conditions that (\ref{E}) cover the (\ref{E0}) staircase support entries are: $n_1$ is chosen arbitrarily and $n_{k+1} \ge 2 ( n_1 + n_2 + \dots + n_k )$. 
\end{theorem}
\begin{proof}
The partition of the matrix for $T$ to blocks goes row by row. We select the sizes for $C_1$ and $A_1$ to be $1\times 1$ and $1 \times 2$, respectively.  This forces $B_1$ to be $2\times 1$. 
For any $k\in\mathbb{N}$, if the sizes of $C_k$, $A_k$ and $B_k$ are $n_k\times n_k$, $n_k\times n_{k+1}$ and $n_{k+1}\times n_{k}$,  respectively, then the members of the sequence $\{n_k\}$ must satisfy the condition $n_{k+1} \ge 2 ( n_1 + n_2 + \dots + n_k )$. Indeed, the width of $A_{k}$ has to be sufficient to cover all support entries on the right of $C_{k}$ and these lie in rows $n_1+n_2+\dots + n_{k-1} +1$ through $n_1+n_2+\dots + n_k$. The last of these support entries stretch out to the position $3(n_1+n_2+\dots + n_k)$. However, block $A_{k}$ starts with the column $n_1+n_2+\dots + n_{k-1} +1$, so by considering  its last (\ref{E0})-staircase row it needs to cover at least $3(n_1+n_2+\dots + n_k)-(n_1+n_2+\dots + n_k)=2(n_1+n_2+\dots + n_k)$ more support entries. Consequently, the $C_{k+1}$ size $n_{k+1} \ge 2 ( n_1 + n_2 + \dots + n_k )$. Taking equality for every $k$ and $n_1=1$ yields (\ref{eq1}).

The necessary and sufficient conditions that (\ref{E}) cover the (\ref{E0}) staircase support entries: $n_1$ is chosen arbitrarily and $n_{k+1} \ge 2 ( n_1 + n_2 + \dots + n_k )$,
follows  by the same argument.
\end{proof}

{
\begin{remark}
Although our canonical choice in Theorem~\ref{T1} uses the minimal value of each $n_k$ for $k\ge 1$, i.e., $n_1=1$ and $n_k=2(n_1+\dots+n_{k-1})$,
the covering of support entries in (\ref{E0}) is \emph{not} minimal in the block tridiagonal sense. That is, if we choose a different sequence $\{n_k'\}$ and denote the 
corresponding blocks in (\ref{E}) by $\{A_n'\}$, $\{B_n'\}$, $\{C_n'\}$, they need not completely cover the full canonical blocks.
For example:  if we select $n_1'=4$, and  $n_{k}'=2(n_1'+\dots+n_{k-1}')$ for $k\ge 2$, the $(4,27)$ entry in the matrix of $T$ does not lie in any of the $\{A_n'\}$, $\{B_n'\}$, $\{C_n'\}$  blocks, yet it belongs to $A_3$.
Since this example has $n_1=4$ one might ask whether the canonical blocks are minimal among those that have $n_1=1$. Once again, the answer is no: take $n_1'=1$, $n_2'=3$, and $n_{k}'=2(n_1'+\dots+n_{k-1}')$ for $k\ge 3$. Again, the $(4,27)$ entry in the matrix of $T$ does not lie in any of the associated blocks, yet it belongs to $A_3$.
\end{remark}
}

\textit{Cyclic vector consequences.}
In case $T$ and $T^*$ have a joint cyclic vector $v$, i.e., the collection $\mathcal{W}(T,T^*)v$ spans $\mathcal{H}$, 
so if in particular the operator $T$ has a cyclic vector $v$ ($\mathcal{W}(T)v$ spans $\mathcal{H}$), then the (\ref{E}) diagonal square blocks have smaller sizes $1\times 1$, $2\times 2$, $4\times 4$, $8\times 8$, $\dots$ and with off-diagonal block sizes forced accordingly. Indeed,
the list $e_1$, $Te_1$, $T^*e_1$, $e_2$, $T^2e_1$, $T^*Te_1$, $e_3$, $TT^*e_1$, ${T^*}^2e_1$, $e_4$, $Te_2$, $T^*e_2$, $e_5$, $\dots$, that was used in the proof of Theorem~\ref{t1} is now substantially reduced. That is, the vectors $e_n$, for $n\ge 2$, can be deleted, their purpose being to ensure that the new orthogonal set spans $\mathcal{H}$. It follows that instead of the $3,6,9,\dots$ pattern we get $2,4,6,8,\dots$. The same reasoning as in the proof of 
Theorem~\ref{T1} now yields the inequality $n_{k+1} \ge  n_1 + n_2 + \dots + n_k $ and equality for all $k$  leads to the blocks of smaller sizes than in (\ref{eq1}). 
Namely from $n_{k+1}=n_1 + n_2 + \dots + n_k $ and $n_k=n_1+n_2+\dots +n_{k-1}$, we obtain $n_{k+1}-n_k=n_k$, $n_{k+1}=2n_k$, and finally $n_{k}=2^{k-1}n_1$.
Then in general, $\mathcal{H}$ can be represented as an orthogonal direct sum of reducing subspaces on which $T$ and $T^*$ have a joint cyclic vector, and thus we have:
\begin{theorem}
Every $\mathcal{B}(\mathcal{H})$ operator is a direct sum of operators of the form (\ref{E}) where the sizes of diagonal blocks in each summand are  $1\times 1$, $2\times 2$, $4\times 4$, $8\times 8$, $\dots$ and the sizes of the off-diagonal blocks are forced accordingly.
\end{theorem}

It is natural to ask whether
Theorem~\ref{T1} can be further improved, i.e.  obtain more universal zeros in the blocks. Here, we are attempting to preserve the structure and block sizes as in  Theorem~\ref{T1}, while ensuring that some  additional entries in these blocks are universally zeros.
The following theorem presents one way to achieve this. It applies more generally only requiring that $\{n_k\}$ be non-decreasing (so that $A_k$ has width no less than its height).

\begin{theorem}\label{T2}
Every block tridiagonal matrix with diagonal block sizes $\{n_k\}$ non-decreasing is unitarily equivalent to a block tridiagonal matrix with the same block sizes but also with $A_n$ of the form $(A'_n \mid 0)$ with $A'_n$ a positive square matrix. Alternatively, the same but with $B_n$ of the form $(B'_n \mid 0)^T$ with $B'_n$ a positive square matrix.
\end{theorem}

\begin{proof}
Consider the block tridiagonal matrix form of $T$ as in (\ref{E}) with $\{n_k\}$ non-decreasing. We will define recursively a sequence of unitary matrices $\{U_n\}$ with the size of $U_n$ same as the size of $C_n$ (i.e., $n_k\times n_k$), and $U$ their direct sum. The matrix for $U^*TU$ becomes
\begin{equation*}
\begin{pmatrix}
U_1^*C_1U_1 & U_1^* A_1 U_2 &0&\dots\\
U_2^*B_1U_1 & U_2^*C_2U_2 &U_2^*A_2 U_3&\ddots\\
0 &U_3^*B_2U_2 &U_3^*C_3U_3 &\ddots \\
\vdots &\ddots &\ddots  &\ddots    
\end{pmatrix}.
\end{equation*}
Let $U_1=I$ and 
suppose that matrices $\{U_i\}_{i=1}^k$ have been selected. Consider the square matrix
$$
X=\begin{pmatrix}
U_k^*A_k\\0
\end{pmatrix},
$$
with the zero matrix of height $n_{k+1}-n_k$, and let $X^*=UP$ be the polar decomposition of $X^*$. Then $P=U^*X^*=XU$, because $P$ is selfadjoint. Further,
$$
XU = \begin{pmatrix}
U_k^*A_k\\0
\end{pmatrix}U = 
\begin{pmatrix}
U_k^*A_kU\\0
\end{pmatrix}.
$$
On the other hand,
$
P^2=PP^*=(XU)(XU)^*=XUU^*X^*=XX^*$, so
$$
P=\left[\begin{pmatrix}
U_k^*A_k\\0
\end{pmatrix}
\begin{pmatrix}
A_k^*U_k&0
\end{pmatrix}\right]^{1/2}=
\begin{pmatrix}
[U_k^*A_kA_k^*U_k]^{1/2}&0\\0&0
\end{pmatrix},
$$
and we define $U_{k+1} =U$. 
This implies that  $U_k^*A_kU_{k+1} =([U_k^*A_kA_k^*U_k]^{1/2}\mid 0)$, hence $T$ is unitarily equivalent to the block tridiagonal form 
\begin{equation}\label{eq33}
\begin{pmatrix}
\tilde C_1&  \tilde A_1  &0 &\dots          \\
\tilde B_1 & \tilde C_2&\tilde A_2 &\ddots             \\       
 0 &\tilde B_2 &\tilde C_3 &\ddots            \\
\vdots &\ddots &\ddots  &\ddots    
\end{pmatrix}
\end{equation}
where each $\tilde A_n$ has the form $(A'_n \mid 0)$ with $A'_n$ a positive square matrix.

The ``alternatively'' part
of Theorem~\ref{T2} follows by applying to $T^*$ the first part. 
\end{proof}

At this point
it is natural to ask whether
this form can be sparsified further. That is, is there a choice of an orthonormal basis  in which (\ref{E})  can be sparsified beyond the above matrix (\ref{eq33})?
Recently, we have been able to prove such a result in the Theorems~\ref{t1}, \ref{T1} setting.
{However we here get the weaker spanning condition:  $e_n\in \vee_{k=1}^{3^n} Ue_k$, than that of Theorem~\ref{t1}.}

\begin{theorem}\label{t3}
For arbitrary $T\in\mathcal{B}(\mathcal{H})$ and any orthonormal basis $\{e_n\}$ of $\mathcal{H}$,
there exists an orthonormal basis $\{f_n\}$  in which $T$ has a block tridiagonal form as in (\ref{E})
with the block sizes as in Theorem~\ref{T1} (with $n_1=1$) and:
\begin{list}{}{\setlength{\leftmargin}{.2in}}
\item (a) each block $B_n=(B_n' \mid 0)^T$ where $B_n'$ is square upper triangular, i.e., $B(i,j)=0$ if $i>j$;
\item (b) each block $A_n$ is 
of the form $(A'_n \mid A_n'' \mid 0)$ where all three blocks are square and
$A_n''$ is 
lower triangular, i.e., $A_n''(i,j)=0$ if $i<j$.
\item (c) $e_1=f_1$ and $e_n\in \vee_{k=1}^{3^n} f_k$ for all $n\in\mathbb{N}$.
\end{list}
{Alternatively, $T$ is unitarily equivalent to another matrix of the form (\ref{E}) with the block sizes as in Theorem~\ref{T1} (with $n_1=1$),
where each $A_n = (A'_n \mid 0)$ and $A_n'$ is square lower triangular, and each $B_n$ has the form $(B'_n \mid B_n''\mid 0)^T$ with $B''_n$ an upper triangular matrix and (c) holds.}
\end{theorem}
\begin{proof}
Our first step is to
define recursively a sequence $\{g_n\}$ that contains  $\{e_n\}$  dispersed more sparsely than the sequence described in the first paragraph of Section 2. 
We will use the notation $n_1=1$,  $n_k=2(3^{k-2})$, for $k\ge 2$, and $s_k=n_1+n_2+\dots+n_k$, for $k\ge 1$ and $s_0=0$. It is clear that every positive integer $n$ can be written in a unique way as 
$$
n=s_k+r,\quad\mbox{ where }\quad 1\le r\le n_{k+1}.
$$
Then we define $g_1=e_1$, $g_2=Te_1$, $g_3=T^*e_1$ and for $n\ge 4$ we use the following formulas:
\begin{gather}
g_n=Tg_{s_{k-1}+r},\quad\mbox{ when }\quad 
s_k+1\le n\le s_k+n_k
.\label{eq21}\\
g_n=T^*g_{r+1-k}\quad\mbox{ when }\quad 
s_k+n_k+1\le n\le s_{k+1}-1
.\label{eq22}\\
g_n=e_k\quad\mbox{ when }\quad 
n= s_{k+1}
.\label{eq23}
\end{gather}
It is not hard to verify that $\{g_n\}$ is a well-defined sequence that contains $\{e_k\}$. 
Let us assume for a moment that $\{g_n\}$ is linearly independent. (The case of linear dependence we deal with at the end of the proof.)
By applying Gram-Schmidt process to $\{g_n\}$ we obtain an orthonormal basis $\{f_n\}$.
An easy calcuation yields $s_{k+1}=3^k$ so the spanning condition follows from (\ref{eq23}).

It follows from (\ref{eq21}) that the length of the nonzero portion of the $m$th column, where $m=s_{k-1}+r$, does not exceed $n=s_k+r$. 
The inequalities in (\ref{eq21}) imply that $1\le r\le n_k$, hence $s_{k-1}+1\le m\le s_{k-1}+n_k=s_k$ and these characterize the columns that go through the block $B_k$. Since the length of the nonzero portions of these columns does not exceed $n=s_k+r = m+n_k$ it follows that $B_k$ is upper triangular and all the blocks below $B_k$ in (\ref{E}) are zeros.

Similarly, the inequalities in
(\ref{eq22}) imply that $n_k+1\le r\le n_{k+1}-1 $, whence
$$
s_{k-1}\le n_k+2-k\le r+1-k\le n_{k+1}-k\le s_{k+1}.
$$
This shows that row $r+1-k$ goes through either the block $A_k$ or $A_{k+1}$. More precisely, it goes through $A_k$ if 
\begin{equation}
\label{eq31}
n_k+2-k\le r+1-k\le s_k,
\end{equation}
and through $A_{k+1}$ if 
\begin{equation}
\label{eq32}
s_k +1\le r+1-k\le n_{k+1}-k.
\end{equation}
If we replace $k$ by $k+1$ in (\ref{eq31}) we obtain
rows numbered 
$n_{k+1}-k+1$ through $s_{k+1}$. Together with (\ref{eq32}) it shows that as $k$ takes positive integer values all rows of the matrix for $T$ appear in (\ref{eq22}).

For rows that go through $A_k$, (resp., $A_{k+1}$) condition in (b) means that the length of the nonzero portion of row $m$ should not exceed $m+2n_{k+1}/3$ (resp., $m+2n_{k+2}/3$).
By (\ref{eq22}), the length of the row $r+1-k$ does not exceed $n=s_k+r = (r+1-k) +(s_k +k-1)$. A calculation  shows that $s_k +k-1\le 2n_{k+1}/3\le 2n_{k+2}/3$ for $k\ge 2$, so (b) is proved together with the fact that all blocks in (\ref{E}) to the right of $A_k$ are indeed zero blocks.

In the case that the sequence $\{g_n\}$ constructed above is not linearly independent, there  exists $n_0\in\mathbb{N}$ for which $g_{n_0}\in \vee_{k=1}^{n_0-1} g_k$. In that case, we will delete the equation that has $g_{n_0}$ on the left side and in the subsequent equations $g_{n_0+1}$ will be replaced by $g_{n_0}$, $g_{n_0+2}$ by $g_{n_0+1}$, etc. Since in each equation 
of the form $g_n=Tg_i$ (resp, $g_n=T^*g_i$), $n$ determines the maximum length of the $i$-th row (resp., column), this will decrease the said maximum by one so the conclusions of the theorem will hold all the more. Of course, if there is a next such number, we apply the same procedure, and so on.

The ``alternatively'' part
of Theorem~\ref{t3} follows by applying to $T^*$ the first part. 
\end{proof}

\begin{remark}
Both Theorem~\ref{T2} and Theorem~\ref{t3} exhibit a lack of symmetry regarding the role of blocks $\{A_n\}$ and $\{B_n\}$, but we do not know if each of these further sparsifications  $A_n,B_n$ forms (Theorems~\ref{T2} and \ref{t3}) can be achieved symmetrically. We suspect not in general.
\end{remark}
\begin{remark}
The results of this section share the same method of finding the desired orthonormal basis $\{f_n\}$. An arbitrary orthonormal basis $\{e_n\}$ is augmented by adding all vectors  of the form $w(T,T^*)e_k$, (all words in ${T,T^*}$, i.e., the free semigroup on two generators, applied to all $e_k$), arranged in a certain order, to which the Gram-Schmidt orthogonalization is applied. It follows that the same results hold even when $T$ is an unbounded operator, as long as all words $w(T,T^*)e_k$ are defined. One condition that achieves this is when each $e_k\in (\mbox{domain }T)\cap (\mbox{domain }T^*)$ and
$(\mbox{range }T)\cup (\mbox{range }T^*)\subset (\mbox{domain }T)\cap (\mbox{domain }T^*)$.
\end{remark}

Both Theorems~\ref{T2} and \ref{t3} show that each block $A_n$ can be sparsified to $(A'_n \mid 0)$ with $A'_n$ a positive square matrix in the former and a lower triangular matrix in the latter. 
(In both cases, the same number of variables have been eliminated.)
Then, one may ask whether there is a further sparsification in which every $A'_n$ is a  diagonal matrix. 
\begin{problem}\label{pbm2.11}
Given an operator $T$, is there 
an orthonormal basis  in which
$T$ is of the form (\ref{E}),
where each $A_n$ has the form $(A'_n \mid 0)$ and $A'_n$ is a diagonal matrix?
\end{problem}
We considered the following $5\times 5$ test matrix (as in Remark~\ref{rem1}):
\begin{equation}
\label{eq3}
T=\begin{pmatrix}
1 &  1 &1 &0  &0   \\
1 & 1&1 &  1 &1    \\   
0 & 1&1 &  1 &1    \\
0 &1&1 &  1 &1    \\
0 &1&1 &  1 &1  
\end{pmatrix}.
\end{equation}
Note: replacing $t_{31}=1$ makes $T$ selfadjoint, hence diagonalizable and not an appropriate test case. Also, 
Remark~\ref{rem1}
shows that one can obtain
the (3,1) entry equal to 0. 
This question was answered  affirmatively by Zack Cramer, University of Waterloo, who produced the following unitary matrix
\begin{equation*}
U=\frac{1}{\sqrt2}\left(\begin{array}{rrrrr}
0 & 0 &\sqrt2 &0  &0   \\
0 & 1&0 &  -1 &0    \\   
0 & 1&0 &  1 &0    \\
1 &0&0 &  0 &1    \\
1 &0&0 &  0 &-1  
\end{array}\right).
\end{equation*}
It is easy to verify that 
\begin{equation*}
U^*TU=\frac{1}{2}\left(\begin{array}{rrrrr}
4 & 4 &0 &0  &0   \\
4 & 4 &\sqrt2 &0  &0   \\   
0 & 2\sqrt2 &2 &  0 &0    \\
0 &0&-\sqrt2 &  0 &0    \\
0 &0&0 &  0 &0  
\end{array}\right),
\end{equation*}
so $T$ is unitarily equivalent to a matrix
with  the entries (1,3), (2,5) and (3,4) equal to 0, and the five zero entries of $T$ remaining zeros. 
Nevertheless, Problem~\ref{pbm2.11} remains open even in the general $5\times 5$ test case.

\end{document}